\documentclass[11pt]{article}
\usepackage{graphicx}
\usepackage{fullpage}
\usepackage{amsmath,amssymb,amsthm}
\usepackage{enumerate}
\usepackage[mathcal]{euscript}
\usepackage{comment}
\usepackage[round]{natbib}

\usepackage[T1]{fontenc}
\usepackage{microtype}
\usepackage{enumitem}

    \numberwithin{equation}{section}

    \usepackage[usenames]{color}
    \definecolor{plum}  {rgb}{.4,0,.4}
    \definecolor{BrickRed} {rgb}{0.6,0,0}
    
    \usepackage{commath}
    \usepackage[normalem]{ulem}

    \usepackage[plainpages=false,pdfpagelabels,colorlinks=true,linkcolor=BrickRed,citecolor=plum]{hyperref}

\def\ddefloop#1{\ifx\ddefloop#1\else\ddef{#1}\expandafter\ddefloop\fi}

\def\ddef#1{\expandafter\def\csname c#1\endcsname{\ensuremath{\mathcal{#1}}}}
\ddefloop ABCDEFGHIJKLMNOPQRSTUVWXYZ\ddefloop

\def\ddef#1{\expandafter\def\csname s#1\endcsname{\ensuremath{\mathsf{#1}}}}
\ddefloop ABCDEFGHIJKLMNOPQRSTUVWXYZ\ddefloop

\def\ddef#1{\expandafter\def\csname b#1\endcsname{\ensuremath{\mathbf{#1}}}}
\ddefloop ABCDEFGHIJKLMNOPQRSTUVWXYZabcdefghijklmnopqrstuvwxyz\ddefloop

\def\Reals{{\mathbb R}}

\def\Naturals{{\mathbb N}}

\def\deq{:=}
\def\eps{\varepsilon}
\def\tp{{\hbox{\textit{\tiny T}}}}

\def\Exp{{\mathbf E}}

    \newtheorem{theorem}{Theorem}[section]
    \newtheorem{lemma}{Lemma}[section]
    \newtheorem{proposition}{Proposition}[section]
    
    \newtheorem{assumption}{Assumption}[section]
    \newtheorem{remark}{Remark}[section]

    \usepackage{cleveref}

\begin{document}

\title{Nonlinear controllability and function representation\\
by neural stochastic differential equations}

\author{Tanya Veeravalli\thanks{University of Illinois; e-mail: {\tt veerava2@illinois.edu}.} \and Maxim Raginsky\thanks{University of Illinois; e-mail: {\tt maxim@illinois.edu}.}}

\date{}

%\editor{}
\maketitle

\begin{abstract}%
    
There has been a great deal of recent interest in learning and approximation of functions that can be expressed as expectations of a given nonlinearity with respect to its random internal parameters. Examples of such representations include ``infinitely wide'' neural nets, where the underlying nonlinearity is given by the activation function of an individual neuron. In this paper, we bring this perspective to function representation by neural stochastic differential equations (SDEs). A neural SDE is an It\^o diffusion process whose drift and diffusion matrix are elements of some parametric families. We show that the ability of a neural SDE to realize nonlinear functions of its initial condition can be related to the problem of optimally steering a certain deterministic dynamical system between two given points in finite time. This auxiliary system is obtained by formally replacing the Brownian motion in the SDE by a deterministic control input. We derive upper and lower bounds on the minimum control effort needed to accomplish this steering; these bounds may be of independent interest in the context of motion planning and deterministic optimal control. 

\end{abstract}

%\begin{keywords} empirical risk minimization, recurrent neural nets, dynamical systems, continuous time, system identification, statistical learning theory, generalization bounds
%\end{keywords}

% !TEX root = neural_SDE_bounds.tex

%%%%%%%%%%%%%%%%%%%%%%%%%%%%%%%%%%%%%%%%%%%%%%%%%%%%%%%%%%%%%%%%%%%%%%%%%%%%%%%%
\section{Introduction}
%%%%%%%%%%%%%%%%%%%%%%%%%%%%%%%%%%%%%%%%%%%%%%%%%%%%%%%%%%%%%%%%%%%%%%%%%%%%%%%%

There has been a great deal of recent interest in learning and approximation of functions that admit continuous representations of the form
\begin{align}\label{eq:continuous_rep}
	F(x) = \int_\Omega \varphi(x; \omega) \mu(\dif \omega),
\end{align}
where $x$ takes values in a domain $\sX$, $\varphi : \sX \times \Omega \to \Reals$ is a structured nonlinearity parametrized by the elements of a parameter space $\Omega$, and $\mu$ is a probability measure on $\Omega$. For example, the seminal paper of \citet{barron1993universal} on neural net approximation uses Fourier-analytic techniques to express a broad class of continuous functions on a compact set $\sX \subset \Reals^d$ in the form  \eqref{eq:continuous_rep} with $\varphi(x; \omega) = \alpha \sigma(\langle \theta, x \rangle + \beta)$, where $\sigma : \Reals \to \Reals$ is a sigmoidal activation function, $\omega = (\alpha,\theta,\beta) \in \Reals \times \Reals^d \times \Reals$, and the probability measure $\mu$ is obtained from the Fourier transform of $f$.

The attractiveness of such a continuous viewpoint, apart from its considerable abstraction power which will be discussed below, is that it affords efficient finite approximation by sampling. The underlying idea, now commonly referred to as \textit{Maurey's empirical method}, is as follows: Let $\sX$, $\Omega$, and $\varphi$ be given. Fix a probability measure $\pi \in \cP(\sX)$ and let $\cF$ denote the class of all $F : \sX \to \Reals$ that can be expressed as in \eqref{eq:continuous_rep} for some $\mu \in \cP(\Omega)$, such that
\begin{align*}%\label{eq:Maurey_var}
V_\pi(F) \deq \int_\sX {\rm var}_{\omega \sim \mu}[\varphi(x;\omega)] \pi(\dif x) < \infty.
\end{align*}
Then a simple probabilistic selection argument shows that, for any $F \in \cF$ and any $N \in \Naturals$, there exist $N$ points $\omega_1,\ldots,\omega_N \in \Omega$, such that the function $\hat{F}_N(\cdot) \deq N^{-1}\sum^N_{i=1}\varphi(\cdot; \omega_i)$ satisfies
\begin{align}\label{eq:Maurey_rate}
	\| F - \hat{F}_N \|^2_{L^2(\pi)} \le \frac{V_\pi(F)}{N}
\end{align} 
--- for $N$ randomly sampled points $\omega_1,\ldots,\omega_N \stackrel{{\rm i.i.d.}}{\sim} \mu$, the expectation $\Exp\| F - \hat{F}_N\|^2_{L^2(\pi)} = N^{-1}V_\pi(F)$, so there exists at least one realization $\omega_1,\ldots,\omega_N$, such that \eqref{eq:Maurey_rate} holds. For instance, in the work of \citet{barron1993universal} and in various extensions \citep{Yukich95universal,Gurvits97approximation,Ji2020Neural} the finite approximation $\hat{F}_N$ is a neural net with one hidden layer consisting of $N$ neurons. 

This idea naturally extends to multilayer neural nets \citep{Barron:2018ty,araujo2019mean,weinan2019barron}, where the elements of the finite-dimensional parameter space $\Omega$ are the weights of all the neurons in the net. One can also consider continuum limits of neural nets, such as neural ordinary differential equations (ODEs) \citep{Hirsch89neural,chen2018neural,li2019deep} and stochastic differential equations (SDEs) \citep{Wong_1991,Tzen2019TheoreticalGF,tzen2019neural,pmlr-v108-li20i}. where the parameter space 
$\Omega$ becomes \textit{infinite-dimensional}. For example, as observed by \citet{Weinan2017dynamical,Haber_2018} and others, one can view the $L$-layer ResNet
\begin{align}\label{eq:ResNet}
	x_{\ell+1} = x_\ell + \frac{1}{L}f(x_\ell,w_\ell), \qquad \ell = 0, \ldots, L-1
\end{align}
with $d$-dimensional input $x_0 = x$, $k$-dimensional weights $w_\ell$, and scalar output $F(x) = \langle \alpha, x_L\rangle$ as an Euler discretization of a controlled ODE $\dot{x}(t) = f(x(t),w(t))$ with initial condition $x(0) = x$ and output $F(x) = \langle \alpha, x(1) \rangle$.  Similar analysis can be given for continuum limits of stochastic neural nets \citep{sonoda2019transport,pmlr-v108-peluchetti20a}. In these settings, the continual representation \eqref{eq:continuous_rep} becomes a \textit{path integral}, where $\Omega = C([0,T]; \Reals^k)$ is the space of continuous paths $\omega : [0,1] \to \Reals^k$ and $\mu$ is a probability measure on $\Omega$ that emerges in the limit as $L \to \infty$ under appropriate assumptions on the distributions of the random vectors $(w_\ell)^{L-1}_{\ell = 0}$ for each $L$.

In this paper, we consider a different class of stochastic models that give rise to path-integral function representations. Instead of \eqref{eq:ResNet}, we start with the continuum limit of the $L$-layer net
\begin{align}\label{eq:StNet}
	x_{\ell+1} = x_\ell + \frac{1}{L}f(x_\ell) + \frac{1}{\sqrt{L}}g(x_\ell)w_\ell, \qquad \ell = 0, \dots, L-1
\end{align}
with $d$-dimensional input $x_0 = x$ and scalar output $F(x) = \Exp[\langle \alpha, x_L \rangle]$,  where $w_0,w_1,\ldots$ are i.i.d.\ standard Gaussian random vectors in $\Reals^d$. The vector $\alpha \in \Reals^d$ and the maps $f : \Reals^d \to \Reals^d$ and $g : \Reals^d \to \Reals^{d \times d}$ are fixed. In contrast to \eqref{eq:ResNet}, the $w_\ell$'s in \eqref{eq:StNet} are not the model weights, but rather serve as a source of stochasticity for generating \textit{random} functions of the initial condition $x$. The recursion \eqref{eq:StNet} is readily recognized as an Euler discretization of the It\^o SDE $\dif X_t = f(X_t)\dif t + g(X_t)\dif W_t$ with drift $f$ and diffusion matrix $g$. We are specifically interested in the case when $f$ and $g$ are elements of some parametric class of functions, thus following the formulation of neural SDEs by \citet{Tzen2019TheoreticalGF,tzen2019neural}. Stated informally, our results show that the functions $F$ that can be represented by models of this kind are of the (approximate) form
\begin{align}\label{eq:F_approx}
	F(x) \approx C_1\int_{\Reals^d} \langle \alpha, y \rangle \exp\big(-C_2 I(x,y)\big) \dif y,
\end{align}
where $c_1,c_2$ are some constants and $I(x,y)$ is a certain \textit{minimum action} functional pertaining to the deterministic optimal control problem of steering the state of the dynamical system $\dot{x}(t) = f(x(t)) + g(x(t))u(t)$ from $x$ to $y$ in finite time by choosing a suitable control $u(\cdot)$. This connection between the existence of densities of diffusion processes and nonlinear controllability  goes back to the pioneering work of \citet{Elliott69thesis}; \citet{sheu1991some} gave sharp quantitative estimates involving $I(x,y)$. Conceptually, the representation of the transition densities of a diffusion process via the minimum action $I(x,y)$ is related to Feynman's path integral approach to quantum mechanics and its classical formulation via optimal control \citep{Fleming83mechanics,Guerra:1983gd}. On the other hand, the representation of functions realizable by neural SDEs in the form \eqref{eq:F_approx} allows us to both characterize their structure and quantify the rates of approximation of such functions by finite sums via Maurey's method. Moreover, some of our results (in particular, upper and lower bounds on the minimum action $I(x,y)$) may be of independent interest in the context of optimal control.

\paragraph{Notation.} The Euclidean inner product and norm on $\Reals^d$ will be denoted by $\langle\cdot,\cdot \rangle$ and by $|\cdot|$, respectively; the spectral norm of a $d \times d$ matrix by $\|\cdot\|$. The $d \times d$ identity matrix is denoted by $I_d$. The Jacobian of a differentiable map $f : \Reals^d \to \Reals^d$ will be denoted by $f_*$. The smallest and the largest eigenvalues of a symmetric matrix will be denoted by $\lambda_{\min}(\cdot)$ and $\lambda_{\max}(\cdot)$ respectively.

%%%%%%%%%%%%%%%%%%%%%%%%%%%%%%%%%%%%%%%%%%%%%%%%%%%%%%%%%%%%%%%%%%%%%%%%%%%%%%%%
\section{The set-up and some background}
\label{sec:setup}
%%%%%%%%%%%%%%%%%%%%%%%%%%%%%%%%%%%%%%%%%%%%%%%%%%%%%%%%%%%%%%%%%%%%%%%%%%%%%%%%

We consider the case when the real-valued function $F(x)$ of a vector-valued $x \in \Reals^d$ is given by a linear functional of the state of an It\^o diffusion process, i.e., $F(x) = \Exp[Z_T|X_0 = x]$ where
\begin{subequations}
\begin{align}
	\dif X_t &= f(X_t)\dif t + g(X_t) \dif W_t, \qquad X_0 = x;\, 0 \le t \le T \label{eq:SDE}\\
	Z_t &= \langle \alpha, X_t \rangle. \label{eq:readout}
\end{align}
\end{subequations}
Here, $T < \infty$ is a fixed time horizon, $(W_t)_{t \in [0,T]}$ is a standard $d$-dimensional Brownian motion, $f : \Reals^d \to \Reals^d$ is the drift, $g : \Reals^d \to \Reals^{d\times d}$ is the diffusion matrix, and $\alpha \in \Reals^d$ is a fixed vector. To obtain a representation of $F$ in the form of \eqref{eq:continuous_rep}, we take $\Omega = C([0,T];\Reals^d)$, the space of continuous paths $\omega : [0,T] \to \Reals^d$; $\mu$ the Wiener measure, under which $(W_t(\omega))_{t \in [0,T]}$ with $W_t(\omega) \deq \omega(t)$ is the standard $d$-dimensional Brownian motion; and $\varphi(x; \omega) \deq \langle \alpha, X_1(x, \omega) \rangle$, where $X_t(x,\omega)$ is the solution of \eqref{eq:SDE} at time $t \ge 0$ with the initial condition $X_0 = x$. With these definitions, we  have
\begin{align}\label{eq:path_integral}
	F(x) = \Exp[Z_1|X_0 = x] = \int_\Omega \varphi(x; \omega)\mu(\dif\omega).
\end{align}
This representation of $F$ as a path integral with respect to the Wiener measure in \eqref{eq:path_integral}, while succinct, is not readily amenable to analysis. However, we can reduce the analysis to finite-dimensional integration if the diffusion process \eqref{eq:SDE} is regular enough to admit transition densities. One such set of regularity assumptions, while by no means the most general, is sufficient for our purposes:

\begin{assumption}\label{as:Sheu} The drift $f(x)$ and the diffusion matrix $g(x)$ satisfy the following:
	\begin{enumerate}
		\item Both $f(\cdot)$ and $a(\cdot) \deq g(\cdot)g(\cdot)^\tp$ are  Lipschitz-continuous.
		\item The diffusion is uniformly elliptic, i.e.,  there exist constants $\lambda_1 \ge \lambda_0 > 0$, such that $\lambda_0 \le \lambda_{\min}(a(x)) \le \lambda_{\max}(a(x)) \le \lambda_1$ for all $x \in \Reals^d$.
	\end{enumerate}
\end{assumption}
\noindent Under these assumptions, there exists a family of functions $p_t : \Reals^d \times \Reals^d \to [0,\infty)$, such that each $p_t(x,\cdot)$ is a probability density, and for any bounded measurable function $h : \Reals^d \to \Reals$ we have
\begin{align*}
	\Exp[h(X_t)|X_0 = x] = \int_{\Reals^d} h(y)p_t(x,y) \dif y, \qquad x \in \Reals^d.
\end{align*}
In fact, more can be said about the regularity of $p_t(x,y)$ as a function of $t$, $x$, and $y$, stemming from the fact that it is the fundamental solution of the parabolic PDE
\begin{align*}
	\frac{\partial}{\partial t} \rho(t,x) = \langle f(x), \nabla_x \rho(t,x) \rangle + \frac{1}{2} {\rm tr}[a(x)\nabla^2_x \rho(t,x)], \qquad t \ge 0, x \in \Reals^d
\end{align*}
i.e., $p_t(x,y) = \rho(t,x)$ with the initial condition $\rho(0,x) = \delta(x-y)$. At any rate, we can now express $F(x)$ in terms of $p_T(x,y)$ as
\begin{align}\label{eq:finite_dim_integral}
	F(x) = \int_{\Reals^d} \langle \alpha, y \rangle p_T(x,y) \dif y
\end{align}
and investigate the questions of expressiveness and rates of approximation by analyzing the structure of $p_t(x,y)$. For instance, classical results in the theory of parabolic PDEs give Gaussian estimates
\begin{align*}
\frac{C_1}{t^{d/2}}\exp\left(-\frac{c_1|y-x|^2}{t}\right) - \frac{C_2}{t^{(d-\kappa)/2}}\exp\left(-\frac{c_2|y-x|^2}{t}\right) \le p_t(x,y) \le \frac{C_3}{t^{d/2}}\exp\left(-\frac{c_3|y-x|^2}{t}\right),
\end{align*} 
where $C_1,C_2,C_3,c_1,c_2,c_3,\kappa$ are some positive constants. (See \citet{Stroock2008PDEs} for these results and relevant background.) However, these bounds are not suitable for our purposes because the dependence of these constants on various parameters (such as the dimension $d$) may not be optimal and because we are specifically interested in the capability of neural SDEs to represent nonlinear functions of $x$. For this reason, we appeal to more fine-grained estimates due to Sheu:

\begin{theorem}[\citet{sheu1991some}]\label{thm:Sheu} Under Assumption~\ref{as:Sheu}, the transition densities $p_T(x,y)$ of \eqref{eq:SDE} can be bounded below and above as
	\begin{align}\label{eq:Sheu_bounds}
		\frac{k_1}{\sqrt{(2\pi T)^d \det a(y)}}  \exp\left(-c_1 I_T(x,y)\right) \le p_T(x,y) \le \frac{k_2}{(\sqrt{(2\pi T)^d \det a(y)}}  \exp\left(-c_2 I_T(x,y)\right)
	\end{align}
	where $k_1,k_2,c_1,c_2$ are polynomial in $d$, $\lambda_1/\lambda_0$, $T$, and the Lipschitz constants of $f$ and $a$, and $I_T(x,y)$ is the minimum action functional for the following deterministic control problem:
	\begin{align}
		I_T(x,y) &\deq \min_{u(\cdot)} \frac{1}{2}\int^T_0 |u(t)|^2 \dif t \nonumber\\
		\text{subject to }& \dot{x}(t) = f(x(t)) + g(x(t))u(t) \label{eq:minimum_action} \\
		&x(0) = x, x(T) = y \nonumber
	\end{align}
\end{theorem}
\noindent Sheu's estimates sharpen and improve earlier results, which are primarily asymptotic and apply to small $T > 0$, small Euclidean (or Riemannian) distance between $x$ and $y$, and small noise. For instance, \citet{Kifer76density} obtains the following Taylor-like expansion under the replacement $g \to \eps g$ for a small $\eps > 0$  under further smoothness conditions on $f,g$, and under the above-mentioned restrictions on $T$, $x$, and $y$:
\begin{align}\label{eq:Kifer}
\begin{split}
	p_T(x,y) &= \frac{1}{\sqrt{(2\pi \eps^2)^{d/2}}} \exp\left(-\frac{1}{2\eps^2}I_T(x,y)\right) \\
	& \quad \cdot \left(K_0(T,x,y) + K_1(T,x,y)\eps^2 + \ldots + K_m(T,x,y)\eps^{2m} + o(\eps^{2m})\right), \quad m = 0, 1, \ldots
\end{split}
\end{align}
where the coefficients $K_i(T,x,y)$ are given by an explicit recursion. The main point here is that, to a reasonable approximation, $p_T(x,y)$ is inversely proportional to the exponential of the minimum action $I_T(x,y)$ (with a multiplicative constant).

%%%%%%%%%%%%%%%%%%%%%%%%%%%%%%%%%%%%%%%%%%%%%%%%%%%%%%%%%%%%%%%%%%%%%%%%%%%%%%%%
\subsection{The role of finite-time nonlinear controllability}
%%%%%%%%%%%%%%%%%%%%%%%%%%%%%%%%%%%%%%%%%%%%%%%%%%%%%%%%%%%%%%%%%%%%%%%%%%%%%%%%

Another lens through which we can view the above observation is that the bounds of Sheu are inversely proportional to the exponential of the minimal ``control energy'' needed to transfer the state of the \textit{deterministic} control-affine system
\begin{align}\label{eq:control_affine}
	\dot{x}(t) = f(x(t)) + g(x(t))u(t)
\end{align}
from $x(0) = x$ to $x(T) = y$. (This system is obtained from \eqref{eq:SDE} by formally replacing the Brownian motion $W$ with a deterministic input $u$.) Such a relationship between the form of the transition density of a diffusion process and deterministic optimal control is, in fact, exact for linear SDEs with $f(x) = Ax$ and $g(x)=G$ for some matrices $A,G \in \Reals^{d \times d}$ \citep{Brockett76nonlinear}. In this case, the transition densities are Gaussian,
\begin{align*}
	p_t(x,y) = \frac{1}{\sqrt{(2\pi t)^d \det W(t)}} \exp\left(-\frac{1}{2}\langle y-e^{tA}x, W(t)^{-1}(y-e^{tA}x)\rangle\right),
\end{align*} 
where the covariance matrix
\begin{align*}
	W(t) = \int^t_0 e^{sA}GG^\tp (e^{sA})^\tp\dif s
\end{align*}
is the time-$t$ controllability Gramian of the linear time-invariant system $\dot{x}(t) = Ax(t) + Gu(t)$ \citep{Brockett1970llinear}. Moreover, the quantity appearing in the exponent is (minus) the minimum action $I_T(x,y)$ defined in \eqref{eq:minimum_action} for this system. Assumption~\ref{as:Sheu} will be satisfied if $GG^\tp$ is positive definite, in which case the controllability Gramian is nonsingular for all $t$. On the qualitative side, a classic result of \cite{Elliott69thesis} says that the complete controllability of the deterministic system \eqref{eq:control_affine} is sufficient for the existence of smooth transition densities of the It\^o diffusion \eqref{eq:SDE}; moreover, \cite{Clark73SDEs} showed that it is also necessary. Remarkably, the result of Sheu (Theorem~\ref{thm:Sheu}) provides a sharp \textit{quantitative statement} of this equivalence for uniformly elliptic diffusions, phrased in terms of the minimum action functional $I_T(x,y)$. (Note that the system \eqref{eq:control_affine} is completely controllable since $g(x)$ has full rank for each $x$ by virtue of Assumption~\ref{as:Sheu}, cf.~for example \citet[Sec.~4.3]{Sontag1998control}).

The importance of this result for our purposes is twofold: First, since Eq.~\eqref{eq:finite_dim_integral} expresses the function $F(x)$ as an expectation of the linear form $\langle \alpha, \cdot \rangle$ w.r.t.\ the transition density $p_T(x,\cdot)$, we conclude that the capability of neural SDEs to represent nonlinear functions hinges on the extent to which the minimum action $I_T(x,y)$ differs from a quadratic form like $\frac{1}{2}\langle y-L(T)x, \Sigma(T)^{-1}(y-L(T)x)\rangle$ for some  $L(T) \in \Reals^{d \times d}$ and a symmetric positive-definite $\Sigma(T) \in \Reals^{d \times d}$. Second, tight upper and lower bounds on $I_T(x,y)$ translate into estimates of the constant $V_\pi(F)$ that controls the dependence of the rate of approximation of $F$ by finite sums, in particular in terms of the dimension $d$ and other model parameters. We turn to the derivation of such bounds next.
 
%%%%%%%%%%%%%%%%%%%%%%%%%%%%%%%%%%%%%%%%%%%%%%%%%%%%%%%%%%%%%%%%%%%%%%%%%%%%%%%%
\section{Upper and lower bounds on the minimum action functional}
\label{sec:I_bounds}
%%%%%%%%%%%%%%%%%%%%%%%%%%%%%%%%%%%%%%%%%%%%%%%%%%%%%%%%%%%%%%%%%%%%%%%%%%%%%%%%

In this section, we obtain upper and lower bounds on the minimum action $I_T(x,y)$, which we need in order to estimate the transition density $p_T(x,y)$. Finite-time optimal transfer of a controlled dynamical system between a prescribed pair of initial and final states is a fundamental primitive in various settings, such as motion planning in robotics \citep{LiCanny1993nonholonomic} or the theory of optimal synthesis \citep{Piccoli_2000}, so these bounds may be of independent interest.

%%%%%%%%%%%%%%%%%%%%%%%%%%%%%%%%%%%%%%%%%%%%%%%%%%%%%%%%%%%%%%%%%%%%%%%%%%%%%%%%
\subsection{Upper bound via feedback linearization}
%%%%%%%%%%%%%%%%%%%%%%%%%%%%%%%%%%%%%%%%%%%%%%%%%%%%%%%%%%%%%%%%%%%%%%%%%%%%%%%%

To obtain an upper bound on $I_T(x,y)$, it suffices to consider any suboptimal control that transfers \eqref{eq:continuous_rep} from $x$ to $y$ in time $T$. A particularly simple way to do this is by \textit{feedback linearization} \citep[Section~5.3]{Sontag1998control}: We use a state feedback control of the form $u = k(x,v)$ so that the closed-loop system $\dot{x}(t) = f(x(t))+g(x(t))k(x(t),v(t))$ becomes, possibly after a smooth and invertible change of coordinates, a completely controllable linear system with input $v(\cdot)$. In our case, since $g(x) \in \Reals^{d \times d}$ is invertible by Assumption~\ref{as:Sheu}, we can take $k(x,v) = g(x)^{-1}(v-f(x))$, resulting in $\dot{x}(t)=v(t)$. This system is evidently completely controllable and, in particular, can be steered from $x(0)=x$ to $x(T)=y$ by means of the constant control $v(t) = \frac{y-x}{T}$. Using this together with uniform ellipticity, we obtain the following simple estimate:

\begin{proposition}
	\begin{align*}%\label{eq:I_upper_bound}
		I_T(x,y) \le \frac{1}{2\lambda_0} \int^T_0 \Big|\frac{y-x}{T} - f\Big(x + \frac{t}{T}(y-x)\Big)\Big|^2 \dif t.
	\end{align*}
\end{proposition}

%%%%%%%%%%%%%%%%%%%%%%%%%%%%%%%%%%%%%%%%%%%%%%%%%%%%%%%%%%%%%%%%%%%%%%%%%%%%%%%%
\subsection{Lower bound via nonlinear variation of parameters}
\label{ssec:var_param}
%%%%%%%%%%%%%%%%%%%%%%%%%%%%%%%%%%%%%%%%%%%%%%%%%%%%%%%%%%%%%%%%%%%%%%%%%%%%%%%%

To obtain a lower bound on $I_T(x,y)$, we fix an arbitrary control $u : [0,T] \to \Reals^d$ and consider two dynamical systems, one with control and one without:
\begin{align*}
	\dot{x}(t) = f(x(t)) + g(x(t))u(t) \qquad \text{and} \qquad \dot{x}(t) &= f(x(t)).
\end{align*}
For $t \ge s$, denote by $\varphi^u_{s,t}(x)$ the state of the first system at time $t$ starting from $x(s)=x$, and define $\varphi_{s,t}(x)$ similarly for the second system. Since the system $\dot{x}(t)=f(x(t))$ is time-invariant, we have $\varphi_{s,t}(\cdot) = \varphi_{0,t-s}(\cdot)$, so we can write $\varphi_{t-s}(\cdot)$ instead. The solutions of the two systems at time $T$ are related by the nonlinear variation-of-parameters formula \citep[p.~96]{Hairer1993ODEs}
\begin{align}\label{eq:var_params}
	\varphi^u_{0,T}(x) - \varphi_T(x) = \int^T_0 (\varphi_{T-t})_*(\varphi^u_{0,t}(x)) g(\varphi^u_{0,t}(x)) u(t) \dif t,
\end{align}
where $(\varphi_{T-t})_*$ denotes the Jacobian of the flow map $\varphi_{T-t}$. We then have the following:

\begin{proposition}\label{prop:I_lower_bound} Under Assumption~\ref{as:Sheu},
	\begin{align}\label{eq:I_lower_bound}
		I_T(x,y) \ge \frac{1}{2\lambda_1 S_T(f)} |y - \varphi_T(x)|^2,
	\end{align}
	where 
	\begin{align}\label{eq:stability_T}
		S_T(f) \deq   \int^T_0 \sup_{\bar{x}\in \Reals^d} \|(\varphi_{T-t})_*(\bar{x})\|^2\dif t.
	\end{align}
\end{proposition}

\begin{proof} Choose any control $u(\cdot)$ that transfers $x$ to $y$ in time $T$ subject to $\dot{x}(t) = f(x(t)) + g(x(t))u(t)$. Since we then have $\varphi^u_{0,T}(x)=y$, using the Cauchy--Schwarz inequality in \eqref{eq:var_params} gives
	\begin{align*}
		|y - \varphi_T(x)|^2 \le \left(\int^T_0 \|(\varphi_{T-t})_*(\varphi_{0,t}(x))\|^2 \dif t\right)   \left(\int^T_0 |g(\varphi^u_{0,t}(x)) u(t)|^2 \dif t \right).
	\end{align*}
Moreover, because $a(\cdot) = g(\cdot)g(\cdot)^\tp \preceq \lambda_1 I_d$, we can further upper-bound this by
	\begin{align*}
		|y - \varphi_T(x)|^2 \le \lambda_1 \cdot  \int^T_0 \sup_{\bar{x} \in \Reals^d} \|(\varphi_{T-t})_*(\bar{x})\|^2 \dif t \cdot \int^T_0 |u(t)|^2 \dif t.
	\end{align*}
	Rearranging and using the definition of $I_T(x,y)$ gives \eqref{eq:I_lower_bound}.
\end{proof}
\noindent Note that the quantity $S_T(f)$ defined in \eqref{eq:stability_T} depends only on the drift $f$ and pertains to stability properties of the the autonomous dynamics $\dot{x}(t) = f(x(t))$. In order to instantiate the lower bound \eqref{eq:I_lower_bound}, we need estimates on $S_T(f)$. To that end, we will make use of the fact that the Jacobian $(\varphi_t)_*(\bar{x})$ of the flow map $\varphi_t$ is the solution $\Lambda(t)$ at time $t$ of the variational equation
\begin{align}\label{eq:vareq}
	\dot{\Lambda}(t) = f_*(\varphi_{t}(\bar{x})) \Lambda(t), \qquad t \ge 0
\end{align}
with the initial condition $\Lambda(0) = I_d$ \citep[Section~2.8]{Sontag1998control}. 

\begin{assumption}\label{as:matmeas} The drift $f$ is such that
	\begin{align*}
		M(f) \deq \sup_{x \in \Reals^d}\lambda_{\max}\Bigg(\frac{f_*(x)+f_*(x)^\tp}{2}\Bigg) < \infty.
	\end{align*}
\end{assumption}
\begin{remark}{\em Note that the quantity $M(f)$ can be positive, negative, or zero.}\end{remark}

\begin{proposition}\label{prop:ST_bound} Under Assumption~\ref{as:matmeas},
	\begin{align}
		S_T(f) \le \begin{cases}
	\frac{1}{2|M(f)|}{e^{2 \max\{M(f),0\} T}} , & M(f) \neq 0 \\
		T, & M(f) = 0
		\end{cases}.
	\end{align}
\end{proposition}

\begin{proof} In order to upper-bound $S_T(f)$, we need uniform control on the operator norms $\| (\varphi_t)_*(\bar{x})\|$ for all $\bar{x} \in \Reals^d$. To that end, we make use of the \textit{measure} (or the \textit{logarithmic norm}) of a matrix $A \in \Reals^{d \times d}$ \citep[Section II.8]{Desoer1975feedback}: Let a norm $\| \cdot \|$ on $\Reals^d$ be given, and let $\| A \| \deq \sup \{ \| Av\| : \|v\|=1\}$ denote the corresponding induced norm of $A$. The measure of $A$ w.r.t.\ $\| \cdot\|$ is then defined as the right directional derivative of $\| \cdot \|$ at $I_d$ in the direction of $A$:
\begin{align*}
	\mu[A] \deq \lim_{\theta \searrow 0} \frac{\| I_d + \theta A\| - 1}{\theta}.
\end{align*}
The following result is key (see, e.g., \citep[p.~34]{Desoer1975feedback}):
\begin{lemma}[Coppel's inequality]\label{lm:Coppel} Let $A : [0,T] \to \Reals^{d \times d}$ be a continuous matrix-valued function. Then any solution of the time-inhomogeneous matrix ODE $\dot{\Lambda}(t) = A(t)\Lambda(t)$ for $0 \le t \le T$ satisfies
	\begin{align*}
		\| \Lambda(t) \| \le \|\Lambda(0)\| \exp\Bigg(\int^t_0 \mu[A(s)]\dif s\Bigg)
	\end{align*}
	where $\mu[A]$ is the measure of $A$ w.r.t.\ the induced norm $\|\cdot\|$.
\end{lemma}
\noindent Since $\mu[A]$ can be positive, negative, or zero, Lemma~\ref{lm:Coppel} is a strict improvement over the usual estimates based on Gr\"onwall's inequality. We will apply it to the case when $\|\cdot\|$ is the spectral norm, i.e., the matrix norm induced by the Euclidean norm $|\cdot|$, so that
$$
\mu[A] = \lambda_{\max}\Big(\frac{A+A^\tp}{2}\Big)
$$
\citep[p.~33]{Desoer1975feedback}. Now fix some $\bar{x} \in \Reals^d$. Since $(\varphi_t)_*(\bar{x})$ is the solution of the variational equation \eqref{eq:var_params} with initial condition $\Lambda(0)=I_d$, Lemma~\ref{lm:Coppel} gives
\begin{align*}
	\| (\varphi_t)_*(\bar{x})\| &\le \exp\left(\int^t_0 \mu[f_*(\varphi_s(\bar{x})] \dif s\right) \\
	&\le \exp\Big(t\sup_{0 \le s \le t} \mu[f_*(\varphi_s(\bar{x})]\Big) \\
	&= \exp\left(t \sup_{0 \le s \le t} \lambda_{\max}\left(\frac{f_*(\varphi_s(\bar{x}))+f_*(\varphi_s(\bar{x}))^\tp}{2}\right)\right) \\
	&\le e^{M(f)t},
\end{align*}
where we have invoked Assumption~\ref{as:matmeas} at the end. Now we estimate $S_T(f)$. When $M(f) \neq 0$,
\begin{align*}
	S_T(f) \le \int^T_0 \sup_{\bar{x}\in\Reals^d} \| (\varphi_{T-t})_*(\bar{x})\|^2 \dif t \le \int^T_0 e^{2M(f)(T-t)}\dif t = \frac{e^{2M(f)T}-1}{2M(f)}.
\end{align*}
When $M(f) = 0$ (e.g., when the Jacobians $f_*(\bar{x})$ are skew-symmetric), $\|(\varphi_t)_*(\bar{x})\| \le 1$ for all $t \ge 0$ and $\bar{x}$, so $S_T(f) \le T$.
\end{proof}
\noindent We should point out that the lower bound of Proposition~\ref{prop:I_lower_bound} did not need any structural properties of the candidate control $u(\cdot)$, other than that it transfers $x$ to $y$ in time $T$. It should be possible to obtain tighter bounds by exploiting the fact that any \textit{optimal} control arises from a Hamiltonian system by Pontryagin's maximum principle \citep{Bonnard2003singular}. We leave this for future work.

%%%%%%%%%%%%%%%%%%%%%%%%%%%%%%%%%%%%%%%%%%%%%%%%%%%%%%%%%%%%%%%%%%%%%%%%%%%%%%%%
\section{Function representation by neural SDEs}
\label{sec:neural_SDEs}
%%%%%%%%%%%%%%%%%%%%%%%%%%%%%%%%%%%%%%%%%%%%%%%%%%%%%%%%%%%%%%%%%%%%%%%%%%%%%%%%

We are now in a position to look at function representation by neural SDEs. As alluded to earlier,  with ``non-Gaussian'' expressions or estimates for the transition density $p_T(x,y)$: (a) we can pinpoint the source of nonlinear dependence of the function $F$ on its input $x$ and (b) we can obtain explicit bounds on the rate of approximation of $F$ by finite sums via Maurey's empirical method.

We begin with the first of these. The key point here, which is primarily qualitative, is that the dominant contribution of the minimum action $I_T(x,y)$ to $p_T(x,y)$ suggests that we should formally think of the function $F$ realized by the neural SDE model
\begin{align}\label{eq:neural_SDE}
	\begin{split}
	\dif X_t &= f(X_t) \dif t + g(X_t) \dif W_t, \qquad X_0 = x \\
	F(x) &= \Exp[\langle \alpha, X_T \rangle | X_0 = x].
	\end{split}
\end{align}
as approximated by expressions like
\begin{align*}
	\hat{F}(x) = C_1\int_{\Reals^d} \langle \alpha, y \rangle \exp\left(-C_2 I_T(x,y)\right)\dif y,
\end{align*}
where the constants $C_1,C_2$ depend on various model parameters; this approximation can be sharpened in appropriate regimes by applying asymptotic expansions for $p_T(x,y)$, such as Eq.~\eqref{eq:Kifer} due to \citet{Kifer76density}.  To a rough first approximation, we can think of $F(x)$ in terms of sampling from a Gaussian density with mean $\varphi_T(x)$ and the inverse of the covariance matrix given by the controllability Gramian of the linearization of the deterministic system $\dot{x}(t) = f(x(t)) + g(x(t))u(t)$ along the trajectory generated by zero control $u(\cdot) \equiv 0$ (this controllability Gramian is referred to as the ``Malliavin covariance matrix'' in \citet{Bismut84Malliavin} due to the deep links between these ideas and Malliavin calculus). Here, as we recall, $\varphi_T(x)$ is the flow map of the autonomous ODE $\dot{x}(t) = f(x(t))$, i.e., just the deterministic contribution of the drift. Moreover, we can iterate the method of Section~\ref{ssec:var_param} to further express $I_T(x,y)$ in terms of $\frac{1}{2}|y-\varphi_T(x)|^2$ plus higher-order terms (see, e.g., \citet{lesiak72volterra}). We can also appeal to recent results on the expressive power of neural ODEs \citep{li2019deep,tabuada2021universal} in order to isolate the contribution of the drift to the expressiveness of neural SDEs. We plan to explore these directions in future work.

Moving on to the rates of approximation, we instantiate Maurey's empirical method by Monte Carlo simulation, i.e., generating $N$ independent copies $X^{x,1}_T,\ldots,X^{x,N}_T$ of $X_T$ in \eqref{eq:neural_SDE} and then building the finite approximation $\hat{F}_N(x) = N^{-1}\sum^N_{i=1} \langle \alpha, X^{x,i}_T \rangle$. Since each $X^{x,i}_T$ is an independent sample from $p_T(x,\cdot)$, for any probability distribution $\pi$ on $\Reals^d$ such that $\varphi_T(\cdot) \in L^2(\pi,\Reals^d)$, we have
\begin{align*}
	\Exp \| F - \hat{F}_N \|^2_{L^2(\pi)} = \int_{\Reals^d} \Exp\Bigg[\Bigg( \langle \alpha, X_T \rangle - \frac{1}{N}\sum^N_{i=1}\langle \alpha, X^{x,i}_T \rangle \Bigg)^2\Bigg] \pi(\dif x) = \frac{V_\pi(F)}{N}, 
\end{align*}
where
\begin{align*}
V_\pi(F) = \int_{\Reals^d}{\rm var}[\langle \alpha, X_T \rangle |X_0 = x]\pi(\dif x).
\end{align*}
We can then estimate $V_\pi(F)$ as follows:

\begin{theorem} Suppose $f$ and $g$ satisfy Assumption~\ref{as:Sheu}. Then
	\begin{align}\label{eq:Maurey_bound}
		V_\pi(F) \le k_2 |\alpha|^2\left(\frac{\lambda_1}{c_2\lambda_0} \frac{S_T(f)}{T}\right)^{d/2} \left(\frac{\lambda_1 d}{c_2}S_T(f) + \int_{\Reals^d} |\varphi_T(x)|^2 \pi(\dif x)\right).
	\end{align}
\end{theorem}
\begin{proof} Using the upper bound in Theorem~\ref{thm:Sheu} and the lower bound in Proposition~\ref{prop:I_lower_bound}, we have
	\begin{align*}
		{\rm var}[\langle \alpha, X_T \rangle |X_0 = x] &\le |\alpha|^2 \int_{\Reals^d} |y|^2 p_T(x,y) \dif y \\
		&\le \frac{k_2|\alpha|^2}{\sqrt{(2\pi \lambda_0 T)^d}}\int_{\Reals^d} |y|^2 \exp\Big(-c_2 I_T(x,y)\Big)\dif y \\
		&\le \frac{k_2|\alpha|^2}{\sqrt{(2\pi \lambda_0 T)^d}}\int_{\Reals^d} |y|^2 \exp\Big(-\frac{c_2}{2\lambda_1 S_T(f)}|y-\varphi_T(x)|^2\Big)\dif y 
	\end{align*}
Evaluating the Gaussian integral and taking expectations w.r.t.\ $\pi$, we get \eqref{eq:Maurey_bound}.
\end{proof}
\noindent If Assumption~\ref{as:matmeas} is also in force, then one can substitute the bounds of Proposition~\ref{prop:ST_bound} for $S_T(f)$. Then $V_\pi(F)$ may or may not scale exponentially with dimension $d$ and the time horizon $T$ depending on the stability properties of the autonomous ODE $\dot{x}(t) = f(x(t))$. Just as in the case of feedforward neural nets \citep{barron1993universal,Ji2020Neural}, this exponential scaling is to be generally expected, apart from various special cases, such as stable linear and nonlinear systems.

%%%%%%%%%%%%%%%%%%%%%%%%%%%%%%%%%%%%%%%%%%%%%%%%%%%%%%%%%%%%%%%%%%%%%%%%%%%%%%%%
\subsection{Example: stochastic recurrent neural nets}
%%%%%%%%%%%%%%%%%%%%%%%%%%%%%%%%%%%%%%%%%%%%%%%%%%%%%%%%%%%%%%%%%%%%%%%%%%%%%%%%

We illustrate the above results in the particular case when
\begin{align*}
	f(x) = - \frac{x}{\tau} + A \boldsymbol{\sigma}(x), \qquad g(x) = cI_d
\end{align*}
where $\tau,c$ are given positive constants, $A \in \Reals^{d \times d}$ is a given matrix, and
$$
\boldsymbol{\sigma}(x) \deq \big( \sigma(x_1), \sigma(x_2), \ldots, \sigma_d(x)\big)^\tp
$$
is a diagonal map constructed from a sigmoidal scalar nonlinearity $\sigma : \Reals \to \Reals$ with $0 \le \sigma'(r) \le \gamma$ for all $r \in \Reals$.  This corresponds to a stochastic recurrent neural net model
\begin{align*}
	\dif X_t &= \left(-\frac{X_t}{\tau} + A\boldsymbol{\sigma}(X_t)\right) \dif t + c \dif W_t, \qquad X_0 = x,\, 0 \le t \le T \\
	F(x) &= \Exp[\langle \alpha, X_T \rangle | X_0 = x].
\end{align*}
It is easy to see that Assumption~\ref{as:Sheu} is satisfied, and $\lambda_0 = \lambda_1 = c^2$.
To verify Assumption~\ref{as:matmeas}, we first write down the Jacobian of $f$:
\begin{align*}
	f_*(x) = -\frac{1}{\tau}I_d + A\, {\rm diag}(\boldsymbol{\sigma}'(x)), \qquad \boldsymbol{\sigma}'(x) \deq \big(\sigma'(x_1),\ldots,\sigma'(x_d)\big)^\tp.
\end{align*}
Following \citet[Section~8]{Hirsch89neural}, we can use the Ger\v{s}gorin disc theorem \citep[p.~244]{Bhatia97matrix} to estimate the eigenvalues of $\frac{1}{2}[f_*(x)+f_*(x)^\tp]$:
\begin{align*}
	M(f) = \sup_{x \in \Reals^d}\lambda_{\max}\left(\frac{f_*(x)+f_*(x)^\tp}{2}\right) \le -\frac{1}{\tau} + \gamma \max_{1 \le i \le d} \Bigg(A_{ii} + \frac{1}{2}\sum_{j:\, j \neq i}(|A_{ij}|+|A_{ji}|)\Bigg).
\end{align*}
For example, if $A_{ii} \le \kappa$, $|A_{ij}| \le \beta$ for all $i,j$ and if every neuron is connected to at most $m \le d$ other neurons, then we can take $M(f) = - \frac{1}{\tau} + \gamma(\kappa + m \beta)$, so different scalings of $S_T(f)$ with $T$ can be achieved by varying the relative magnitudes of the net parameters.

%%%%%%%%%%%%%%%%%%%%%%%%%%%%%%%%%%%%%%%%%%%%%%%%%%%%%%%%%%%%%%%%%%%%%%%%%%%%%%%%
\section*{Acknowledgments}
This work was supported in part by the Illinois Institute for Data Science and Dynamical Systems (iDS$^2$), an NSF HDR TRIPODS institute, under award CCF-1934986.
%%%%%%%%%%%%%%%%%%%%%%%%%%%%%%%%%%%%%%%%%%%%%%%%%%%%%%%%%%%%%%%%%%%%%%%%%%%%%%%%

\bibliography{neural_SDE_bounds.bbl}

\end{document}